\documentclass[12pt]{amsart}

\usepackage{amssymb}
\usepackage{graphicx}
\usepackage{enumerate}
\usepackage{multirow}
\usepackage{amsmath,color}
\usepackage{hyperref}
\usepackage{url}
\usepackage[colorinlistoftodos]{todonotes}

\newtheorem{thm}{Theorem}[section]
\newtheorem{cor}[thm]{Corollary}
\newtheorem{lem}[thm]{Lemma}
\newtheorem{prop}[thm]{Proposition}

\numberwithin{equation}{section}
\theoremstyle{definition}

\def\Cay{\mathop{\rm Cay }\nolimits}
\def\Dih{\mathop{\rm Dih }\nolimits}
\newcommand{\seq}[1]{\langle #1\rangle}

\title[Bipartite distance-regular Cayley graphs]{On bipartite distance-regular Cayley graphs with small diameter}

\author{Edwin R. van Dam}
\address{Department of Econometrics and O.R., Tilburg University, The Netherlands}
\email{Edwin.vanDam@uvt.nl}
\author{Mojtaba Jazaeri}
\address{Department of Mathematics, Shahid Chamran University of Ahvaz, Ahvaz, Iran}
\address{School of Mathematics, Institute for Research in Fundamental Sciences (IPM), P.O. Box: 19395-5746, Tehran, Iran}
\email{M.Jazaeri@scu.ac.ir, M.Jazaeri@ipm.ir}
\begin{document}




\begin{abstract}
We study bipartite distance-regular Cayley graphs with diameter three or four. We give sufficient conditions under which a bipartite Cayley graph can be constructed on the semidirect product of a group --- the part of this bipartite Cayley graph which contains the identity element --- and $\mathbb{Z}_{2}$. We apply this to the case of bipartite distance-regular Cayley graphs with diameter three, and consider cases where the sufficient conditions are not satisfied for some specific groups such as the dihedral group.
We also extend a result by Miklavi\v{c} and Poto\v{c}nik that relates difference sets to bipartite distance-regular Cayley graphs with diameter three to the case of diameter four. This new case involves certain partial geometric difference sets and --- in the antipodal case --- relative difference sets.

\end{abstract}

\maketitle

\section{Introduction}

We study bipartite distance-regular Cayley graphs with diameter three or four.
The subjects of both distance-regular graphs and Cayley graphs form important areas in algebraic graph theory. For background (and more) on distance-regular graphs, we refer to the monograph \cite{BCN}, survey \cite{DKT}, and website \cite{drgorg}. The question which distance-regular graphs are Cayley graphs is a problem which has received increasing attention recently (see \cite[problem~70]{DKT}). Miklavi\v{c} and Poto\v{c}nik \cite{MP2} for example classified distance-regular Cayley graphs on dihedral groups. It turned out that every non-trivial distance-regular Cayley graphs on a dihedral group is bipartite with diameter three. This gives rise to the question which bipartite distance-regular graphs with diameter three are Cayley graphs. Is it the case that all such distance-regular Cayley graphs can be realized on the semidirect product of a group and $\mathbb{Z}_{2}$? We note that, in general, for a given bipartite graph, it is an NP-complete problem to decide if it has an automorphism of order $2$ which interchanges the two parts; see \cite[p.~106]{Ba}.

In this paper, we give sufficient conditions under which a bipartite Cayley graph can be constructed on the semidirect product of a group --- the part of this bipartite Cayley graph which contains the identity element --- and $\mathbb{Z}_{2}$. We apply this to the case of bipartite distance-regular Cayley graphs with diameter three, and consider cases where the sufficient conditions are not satisfied for some specific groups such as the dihedral group.

The development of a difference set is well known to be a symmetric design. In turn, the incidence graph of a symmetric design is a bipartite distance-regular graph with diameter three. Miklavi\v{c} and Poto\v{c}nik \cite{MP2} made precise when this construction from a difference set leads to a Cayley graph. Here we extend this to the case of diameter four. This new case involves certain partial geometric difference sets and --- in the antipodal case --- relative difference sets.

We note that Chen and Li \cite{CL} studied relative difference sets in relation to antipodal distance-regular graphs with diameter three. In this case the connection set of the Cayley graph is the relative difference set, just like partial difference sets are connection sets of strongly regular graphs (i.e., distance-regular graphs with diameter two).

The paper is further organized as follows. In Section \ref{Sec:drCgandDS}, we introduce notation and the relation between distance-regular graphs with diameter three and difference sets. We recall the above mentioned correspondence between difference sets and Cayley graphs by Miklavi\v{c} and Poto\v{c}nik \cite{MP2} in Proposition \ref{Bipartite distance-regular and difference set}. We add to this an observation about constructing bipartite Cayley graphs from sets in an abelian group in Proposition \ref{lemma:devCayley}.

In Section \ref{sec:bipartiteCayley}, we consider bipartite Cayley graphs. We give conditions on the order and size such that a bipartite Cayley graph can be constructed on the semidirect product of a group --- the part of this bipartite Cayley graph which contains the identity element --- and $\mathbb{Z}_{2}$. In Proposition \ref{zeroeigenvalue}, we extend this result to the case when the graph has no eigenvalue $0$. Before applying all this in Section \ref{sec:d3}, we first introduce certain partial geometric difference sets in Section \ref{PGDS}, as these are relevant both for diameters three and four. In Proposition \ref{prop:0mupgds}, we extend the result by Miklavi\v{c} and Poto\v{c}nik \cite{MP2} by relating certain partial geometric difference sets to bipartite distance-regular Cayley graphs with diameter four. We apply this to some examples before zooming in on the antipodal bipartite case in Proposition \ref{propRelDifSet}. Also here we mention some interesting examples.

In Section \ref{sec:d3}, we then apply the results of Section \ref{sec:bipartiteCayley} to the case of bipartite distance-regular Cayley graphs with diameter three, and consider some specific groups such as the dihedral group.

In the final section, we mention some (mostly known) results on bipartite distance-regular Cayley graphs of larger diameter.

\section{Distance-regular Cayley graphs and difference sets}\label{Sec:drCgandDS}

\subsection{Preliminaries}

Let $\Gamma$ be a connected graph with diameter $d$. Then $\Gamma$ is called distance-regular with intersection array $\{b_{0},b_{1},\ldots,b_{d-1};\break c_{1},c_{2},\ldots,c_{d}\}$ whenever, for each pair of vertices $x$ and $y$ at distance $i$, where $i=0,1,\ldots,d$, the number of neighbours of $x$ at distance $i+1$ and $i-1$ from $y$ are constant numbers $b_{i}$ and $c_{i}$, respectively. This implies that a distance-regular graph is regular with valency $b_{0}=k$ and that the number of vertices at distance $i$ from a fixed vertex is constant. This number is denoted by $k_{i}$ and it follows that $k_{i+1}=\frac{k_{i}b_{i}}{c_{i+1}}$, where $i=0,1,\ldots,d-1$. Also the number of neighbours of $x$ at distance $i$ from $y$ is a constant number $k-b_{i}-c_{i}$, which is denoted by $a_{i}$. If $\Gamma$ is bipartite, then $a_i=0$ for  $i=0,1,\ldots,d$.

Let $G$ be a finite group and $S$ be an inverse-closed subset of $G$ not containing the identity element; we call $S$ the connection set. Then the Cayley graph $\Cay(G,S)$ is the graph whose vertex set is $G$, where two vertices $a$ and $b$ are adjacent (denoted by $a \sim b$) whenever $ab^{-1} \in S$. By $S_i$ we denote the set of elements at distance $i$ in $\Cay(G,S)$ from the identity element. Recall that a graph $\Gamma$ is a Cayley graph if and only if there exists a subgroup of the automorphism group of $\Gamma$ which acts regularly on the vertex set of $\Gamma$ (see e.g.,~\cite[Lemma~16.3]{B}). In this paper, the identity element of a group $G$ is denoted by $e$ and its order by $|G|$.

A semidirect product of a group $H$ with a group $K$, which is denoted by $H \rtimes K$ or $K \ltimes H$, is a (not necessarily unique) group $G$ containing a normal subgroup $H_{1}$ which is isomorphic to $H$ and a subgroup $K_{1}$ isomorphic to $K$ such that $G=H_{1}K_{1}$ and $H_{1} \cap K_{1}=\{e\}$.

\subsection{Symmetric designs and difference sets}

A $2$-$(n,k,\mu)$ design consists of a finite set of order $n$ (of elements called points) and a family of $k$-element subsets (with $k \geq 2$) of this set (called blocks) such that each pair of points is included in exactly $\mu$ blocks. Moreover, this design is called symmetric whenever the numbers of points and blocks are equal. The incidence graph of a symmetric $2$-$(n,k,\mu)$ design is a bipartite graph with two parts of points and blocks such that a point is adjacent to a block whenever the point lies in the block. It is known that this incidence graph is a bipartite distance-regular graph with diameter three, and the other way around.

Indeed, let $\Gamma$ be a bipartite distance-regular graph with diameter $3$. Then its intersection array is $\{k,k-1,k-\mu;1,\mu,k\}$, where $\mu = c_{2}$; see also the distance-distribution diagram below. It is well known that the graph $\Gamma$ is the incidence graph of a symmetric $2$-$(n,k,\mu)$ design, where $2n$ is the number of vertices of this graph (see e.g.,~\cite[Thm.~5.10.3]{Go}). Note that if such a design exists, then $k(k-1)=(n-1)\mu$, for example because $n-1=k_2=\frac{k(k-1)}{\mu}$. If $k<n-1$ (in which case we call the graph non-trivial), then the distance-$3$ graph $\Gamma_{3}$ is also a bipartite distance-regular graph with intersection array $\{n-k,n-k-1,k-\mu;1,n-2k+\mu,n-k\}$, which is the incidence graph of the so-called complementary symmetric $2$-$(n,n-k,n-2k+\mu)$ design. Furthermore, $\{k^{1},\sqrt{k-\mu}^{[n-1]},-\sqrt{k-\mu}^{[n-1]},-k^{1}\}$ is the spectrum (the multiset of all eigenvalues of the adjacency matrix) of this graph.
\begin{center}
\begin{figure}[h]

 \centering
 \centerline{\includegraphics*[height=2cm]{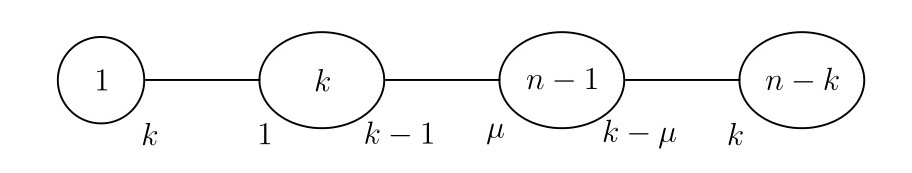}}
 \label{diagram}
 \caption{Distance-distribution diagram}
\end{figure}
\end{center}

Trivial examples in the context of this paper are the complete bipartite graph $K_{n,n}$ minus a perfect matching. These are the incidence graphs of the (trivial) $2$-$(n,n-1,n-2)$ designs; or the bipartite distance-regular graphs with intersection array $\{k,k-1,1;1,k-1,k\}$, i.e., those with $k=n-1$, or equivalently $\mu=k-1$. Such graphs are Cayley graphs on the dihedral group $D_{2n}$ (see e.g.,~\cite[\S3.1]{VJ}).

An $(n,k,\mu)$-difference set in a finite group $H$, where $|H|=n$, is a subset $D \subseteq H$ of size $k$ such that every non-identity element of $H$ can be expressed exactly $\mu$ times as (a ``multiplicative difference") $d_{1}d_{2}^{-1}$, where $d_{1}$, $d_{2}$ are elements of $D$. Moreover, from an $(n,k,\mu)$-difference set $D$ in a finite group $H$, one can construct a symmetric $2$-$(n,k,\mu)$ design called the development of $D$, by considering $\{Dh \mid h \in H\}$ as the set of blocks. We note that if $D$ is a difference set, then $H \setminus D$ is also a difference set, and its development is the complementary design of the development of $D$. Moreover, $D^{-1}$ is also a difference set (even when $H$ is nonabelian), because its development gives the dual design (i.e., the design obtained by interchanging the role of points and blocks). An $(n,k,\mu)$-difference set $D$ is called trivial whenever $k \in \{0,1,n-1,n\}$.

We recall the following result by Miklavi\v{c} and Poto\v{c}nik \cite{MP2} that relates bipartite distance-regular Cayley graphs with diameter $3$ and certain difference sets.

\begin{prop} \cite[Lemma~2.8]{MP2} \label{Bipartite distance-regular and difference set}
Let $G$ be a group of order $2n$ and $S$ a subset of $G$ of size $k$. Then the following statements are equivalent:\footnote{In the formulation in \cite{MP2} we replaced $D=a^{-1}S$ by $D=Sa^{-1}$ because of our slightly different definition of Cayley graphs. The original formulation is also valid with our definition, but it is less natural.}
\begin{itemize}
\item $S \subseteq G \setminus \{e\}$, $S = S^{-1}$ and $\Cay(G,S)$ is a non-trivial bipartite distance-regular graph with diameter $3$ and intersection array $\{k,k-1,k-\mu;1,\mu,k\}$;
\item there is a subgroup $H$ of index $2$ in $G$ and an element $a \in G \setminus H$ such that the set $D=Sa^{-1}$ is a non-trivial $(n,k,\mu)$-difference set in $H$ satisfying $D^{-1}=aDa$;
\item there is a subgroup $H$ of index $2$ in $G$ such that for every $a \in G \setminus H$, the set $D=Sa^{-1}$ is a non-trivial $(n,k,\mu)$-difference set in $H$ satisfying $D^{-1}=aDa$.
\end{itemize}
\end{prop}

For abelian groups, we can be a bit more specific (and constructive).

\begin{prop}\label{lemma:devCayley}
Let $D$ be a subset (not necessarily a difference set) of an abelian group $H$. Then the incidence graph of its development is isomorphic to the Cayley graph on the generalized dihedral group $G=\Dih(H)=H \rtimes \mathbb{Z}_{2}$ with connection set $S=Dc$, where $c^2=1$ ($c\notin H$) and $chc=h^{-1}$ for every element $h\in H$.
\end{prop}

\begin{proof}
We note that $S$ contains only elements of order $2$ and hence it is inverse-closed. The map that sends (point) $h$ to $h$ and (block) $Dh$ to $h^{-1}c$ is an isomorphism. Indeed, if $h_1 \sim Dh_2$ (in the incidence graph), then $h_1h_2^{-1} \in D$. It follows that
$h_1(h_2^{-1}c)^{-1}=h_1ch_2=h_1h_2^{-1}c \in Dc$ and therefore $h_1$ is adjacent to $h_2^{-1}c$ in the Cayley graph, and the other way around.
\end{proof}

This lemma applies for example in the case of the incidence graph of a Desarguesian projective plane (as we observed before in \cite[\S~3.5]{VJ}). Recall that a projective plane of order $q$ is a symmetric $2$-$(q^{2}+q+1,q+1,1)$ design
and the incidence graphs of projective planes are precisely the bipartite distance-regular graphs with diameter three and girth $6$ (i.e., with $c_2=1$); these have intersection array $\{k,k-1,k-1;1,1,k\}$, with $k=q+1$.
Indeed, it is well known that a Desarguesian projective plane can be constructed as the development of a (Singer) difference set in a cyclic group, hence by the above lemma, its incidence graph is indeed a Cayley graph, more specifically on a dihedral group. We conclude the following (well-known) result.
\begin{prop}
Let $\Gamma$ be a bipartite distance-regular graph with diameter $3$ and girth $6$. Then $\Gamma$ is the incidence graph of a projective plane. Moreover, if the projective plane is Desarguesian, then $\Gamma$ is a Cayley graph on a dihedral group.
\end{prop}

For a table of abelian difference sets with small parameters, we refer to \cite[Table 18.73]{CD}.

We finally note that there are non-abelian difference sets for which the development is not a Cayley graph. Before we go deeper into distance-regular Cayley graphs, we now first derive some elementary results on bipartite Cayley graphs.

\section{Bipartite Cayley graphs}\label{sec:bipartiteCayley}

In this section, we will first consider some elementary properties of bipartite Cayley graphs, mostly motived by Proposition \ref{lemma:devCayley}. In Section \ref{sec:nonsingular}, we will go deeper into the case of nonsingular graphs, which will all be used in Section \ref{sec:d3} when we continue with the study of bipartite distance-regular Cayley graphs with diameter $3$.

\subsection{Normal subgroups and the semidirect product}

\begin{lem} \label{Normal subgroup}
Let $\Cay(G,S)$ be a connected bipartite Cayley graph. Then the part of this bipartite graph which contains the identity element is a normal subgroup $H$ of index $2$ in the group $G$ and $S \subseteq G \setminus H$.
\end{lem}
\begin{proof}
Let $H$ be the part of the bipartite graph $\Cay(G,S)$ which contains the identity element $e$. If $a,b \in H$, then $ba \in H$. To see this let $d(e,b)=2n$, where $n \in \mathbb{N}$. Then there exists a path of length $2n$ between $a$ and $ba$ which implies that $ba \in H$. Moreover, similarly, $a^{-1},b^{-1} \in H$. This implies that $H$ is a subgroup of $G$. On the other hand, each part of this regular bipartite graph has the same size and therefore the subgroup $H$ is normal since the index of this subgroup in $G$ is $2$ and this completes the proof.
\end{proof}

\begin{lem} \label{oddparameters}
Let $\Gamma=\Cay(G,S)$ be a connected bipartite Cayley graph of order $2n$ and valency $k$ and let $H$ be the part of this bipartite graph which contains the identity element. If $n$ is odd or $k$ is odd, then $G$ is isomorphic to $H \rtimes \mathbb{Z}_{2}$.
\end{lem}

\begin{proof} First, let $n$ be an odd number. Then the normal subgroup $H$ has odd order $n$. Because $G$ has even order, it contains an involution $a$, which clearly cannot be in $H$. Therefore $G = H \rtimes \seq{a}$.

Secondly, let $k$ be an odd number, so $|S|$ is odd.  Because $S=S^{-1}$, the number of $s\in S$ that is not an involution (i.e., for which $s\neq s^{-1}$) is even, so there must be an involution $a \in S$. Therefore $G = H \rtimes \seq{a}$.
\end{proof}

\subsection{Complete bipartite graphs}
Let us first make a few observations about the bipartite distance-regular graphs with diameter two: the regular complete bipartite graphs $K_{n,n}$, with $n>1$. These are Cayley graphs for any group $G$ of order $2n$ having a subgroup $H$ of order $n$, by considering $S=G\setminus H$. If moreover $S$ contains an involution $a$, then $G = H \rtimes \seq{a}$.
For odd $n$, it is clear from the above that this is the case. However, for even $n$, it is different. For example if $G$ is the (abelian) group $G =\mathbb{Z}_{2n}$, which has a (abelian) subgroup $H$ (of even numbers) isomorphic to $\mathbb{Z}_{n}$, but $G \setminus H$ does not have an involution. Still, $K_{n,n}$ can easily be constructed on a semidirect product $H \rtimes \mathbb{Z}_{2}$ such as the dihedral group. We will also see this distinction for bipartite distance-regular graphs with larger diameter.

\subsection{No involutions}

Let us try to find a general setting for the previous example, i.e., in the case that there are no involutions in $G\setminus H$. Consider a bipartite Cayley graph $\Gamma=\Cay(G,S)$ where the usual subgroup $H$ is such that a semidirect product of $H$ and $\mathbb{Z}_2$ can be defined (such as when $H$ is abelian). For $a \in G \setminus H$, we let $T_a=Sa^{-1}$, which is a subset of $H$. Let $G' = H \rtimes \seq{c}$, with $c^2=e$, but $c \notin G$. We now take $S'=T_ac$. If $S'$ is inverse-closed (in $G'$), then we can use it as a connection set and obtain a bipartite Cayley graph $\Gamma'=\Cay(G',S')$.

\begin{lem} \label{noinvolutions}
The Cayley graph $\Gamma'=\Cay(G',S')$ is isomorphic to $\Gamma=\Cay(G,S)$.
\end{lem}

\begin{proof} The map $\varphi:\Cay(G,S) \rightarrow \Cay(G',S')$ defined by $$\varphi(h)=h, \varphi(a^{-1}h)=ch,$$ for $h \in H$, is an isomorphism, as is easily checked.
\end{proof}

As a corollary, we obtain the following (cf.~Proposition \ref{lemma:devCayley}).
\begin{cor} \label{gendihedral}
Let $\Gamma=\Cay(G,S)$ be a connected bipartite Cayley graph and let $H$ be the part of this bipartite graph which contains the identity element. If $H$ is abelian, then  $\Gamma$ can be constructed as a Cayley graph on the generalized dihedral group $\Dih(H)$.
\end{cor}

\begin{proof} We can apply the above construction with $G'=\Dih(H)$ for any $a \in G \setminus H$. The set $S'$ is inverse-closed because all elements in $\Dih(H) \setminus H$ are involutions.
\end{proof}

A particular example where this applies (and that is more interesting than the complete bipartite graph) is the following description as a Cayley graph of the $4$-cube (a bipartite distance-regular graph with diameter $4$); the only $4$-regular bipartite Cayley graph on $16$ vertices with integral eigenvalues, according to Minchenko and Wanless \cite{MW}.

Let $G=\mathbb{Z}_{4} \times \mathbb{Z}_{4}=\langle a,b \mid a^{4}=b^{4}=1,ab=ba \rangle$ and $S=\{a,a^{-1},b,b^{-1}\}$. Then $H=\langle ab,a^{2} \rangle$, which is isomorphic to $\mathbb{Z}_{4} \times \mathbb{Z}_{2}$, but $G \setminus H$ does not have involutions.
However, the $4$-cube can also be described on $\Dih(H)$, or $(\mathbb{Z}_{4} \times \mathbb{Z}_{2}) \rtimes \mathbb{Z}_2=\langle ab,a^{2},c \mid (ab)^{4}=( a^{2} )^{2}=c^{2}=1,(ab) a^{2} = a^{2} (ab),(c(ab))^{2}=(c a^{2} )^2=1\rangle$, with connection set $S' = Sa^{-1}c$.

We note that the above construction is not the only way to obtain isomorphic bipartite Cayley graphs. For example, consider the below bipartite Cayley graph on $18$ vertices and valency $4$. Note that this is the only such Cayley graph with integral eigenvalues \cite{MW}; it is the bipartite double of the Paley graph $P(9)$.

Indeed, let $G=\mathbb{Z}_{6} \times \mathbb{Z}_{3}=\langle a,b \mid a^{6}=b^{3}=1,ab=ba \rangle$, $S=\{a,a^{5},a^{3}b,a^{3}b^{2}\}$, $G'=(\mathbb{Z}_{3} \times \mathbb{Z}_{3})\rtimes \mathbb{Z}_{2}=\langle a^{2},b,c \mid ( a^{2} )^{3}=b^{3}=c^{2}=1, a^{2}b=ba^{2} ,ca^{2}c=a^{-2},cbc=b^{-1} \rangle$, and $S'=\{c,a^{2}c,bc,a^{2}bc\}$. The normal subgroup $H$ in $G$ is isomorphic to $\mathbb{Z}_{3} \times \mathbb{Z}_{3}$ but it can be checked that there is no element $g \in G\setminus H$, such that $S'=Sg^{-1}c$. Still, the corresponding Cayley graphs are isomorphic.

\subsection{Bipartite nonsingular Cayley graphs}\label{sec:nonsingular}

In this section, we consider bipartite Cayley graphs that have no eigenvalues $0$, in order to obtain similar results as in the previous section for the cases that $k$ is even and $n$ is even, but not divisible by $4$. Note that bipartite distance-regular graphs with odd diameter have no eigenvalue $0$ (contrary to those with even diameter).

\begin{prop} \label{zeroeigenvalue}
Let $\Gamma=\Cay(G,S)$ be a connected bipartite Cayley graph of order $2n$ and valency $k$ and let $H$ be the part of this bipartite graph which contains the identity element. If $n$ is not divisible by $4$ and $\Gamma$ has no eigenvalue $0$, then $G$ is isomorphic to $H \rtimes \mathbb{Z}_{2}$.
\end{prop}

\begin{proof} By Lemma \ref{oddparameters}, we may assume that $n$ is even. Assume that $n$ is not divisible by $4$, so that $n=2m$ for some odd $m$. Then the Sylow $2$-subgroup $L$, say, of $G$ has order $4$ and $L$ is not contained in the normal subgroup $H$ (the latter which has order $n$).

If $L$ is cyclic, then we may use a so-called normal $p$-complement theorem. Indeed, by \cite[Cor.~10.24]{Rose}, there exists a normal subgroup $N$ of (odd) order $m$ in the group $G$ such that $G=NL$ and $N \cap L =\{e\}$. Now the quotient group $G/N$ is cyclic because $L$ is cyclic, so $G/N=\{N,Na,Na^{2},Na^{3}\}$ for some $a \notin N$.
Note also that $N$ is contained in $H$, for otherwise $G=H \cup gH$ for some $g \in N \setminus H$, so $G=NH$, which implies that $|G|=8|H \cap N|$, which is a contradiction.
Because we may assume that the involution $a^2$ is in $H$, it follows that $H=N \cup Na^{2}$.

Next, we will use that the cosets of a normal subgroup form an equitable partition in a Cayley graph \cite[\S~2.3]{VJ}. In this case, we obtain an equitable partition with four parts $G/N$ for $\Gamma$. Moreover, we have a very particular quotient matrix. If $xa \in S$, where $x \in N$, then $(xa)^{-1} \in Na^{3}$. This implies that the numbers of adjacent vertices to a fixed vertex ($x$) of $N$ in $Na$ and in $Na^{3}$ are equal. It follows that the quotient matrix equals $k/2$ times the adjacency matrix of the $4$-cycle, which has eigenvalue $0$. Thus, also $\Gamma$ has eigenvalue $0$ \cite[Lemma~2.3.1]{BH}, which contradicts our assumption. Hence the Sylow $2$-subgroup $L$ cannot be cyclic.

Thus, $L$ is isomorphic to $\mathbb{Z}_{2} \times \mathbb{Z}_{2}$, and $L\setminus H$ (which is nonempty) contains an involution $a$, and it follows that $G = H \rtimes \seq{a}$.
\end{proof}

\section{Partial geometric difference sets}\label{PGDS}

Before we apply the results of Section \ref{sec:bipartiteCayley}, we first need to introduce partial geometric designs and difference sets. Particular cases of these are also defined, as they naturally connect to bipartite distance-regular graphs with diameter $4$.

\subsection{Partial geometric designs and distance-regular graphs with diameter $4$}

A symmetric partial geometric design --- or symmetric $1\frac{1}{2}$-design --- with parameters $(n,k,\alpha,\beta)$ is a $1$-design with $n$ points and $n$ blocks of size $k$ with the property that for each point-block pair $(p,B)$, the number of incident point-block pairs $(p',B')$, $p'\neq p$, $B' \neq B$, with $p' \in B$ and $p \in B'$ equals $\beta$ or $\alpha$, depending on whether $p$ is in $B$ or not, respectively.
Note that every symmetric design is also a partial geometric design.

The incidence graphs of symmetric partial geometric designs are precisely the regular bipartite graphs with four or five distinct eigenvalues \cite{VS}. Not all regular bipartite graphs with five distinct eigenvalues are distance-regular, as in the case of four eigenvalues. It was shown however \cite[Prop.~3.5]{AlmostDRG} that if in such a graph, the number of common neighbors of two vertices at distance $2$ is constant ($c_2=\mu$), then it is distance-regular (with diameter $4$).

On the other hand, it is known (see \cite[Prop.~1.7.1]{BCN}) that a bipartite distance-regular graphs with diameter $4$ is the incidence graph of an incidence structure called (square) partial $\lambda$-geometry, as introduced by Drake \cite{Drake}. We need not further define these but instead build on the definition of (symmetric) partial geometric design with the additional property that any two points meet in either $0$ or $\mu$ blocks, and dually, any two blocks share either $0$ or $\mu$ points.\footnote{We denote the $\lambda$ in partial $\lambda$-geometry by $\mu$.}

In general, a bipartite distance-regular graph $\Gamma$ with diameter $4$ on $2n$ vertices has intersection array $\{k,k-1,k-\mu,k-c_3;1,\mu,c_3,k\}$, so $\mu = c_{2}$, and $c_3=\frac{k(k-1)(k-\mu)}{(n-k)\mu}$ (which follows from the standards relations between parameters). 
The halved graphs of $\Gamma$ are strongly regular graphs, with parameters following from the above intersection array. Recall that if $\Gamma$ is a Cayley graph, then also these halved graphs are Cayley graphs, which may give extra restrictions for existence as a Cayley graph.

An example where this applies is the distance-regular graph on $100$ vertices (and a related partial $5$-geometry on 50 points) that can be constructed from the cocliques in the Hoffman-Singleton graph (see \cite[Thm.~13.1.1(iv)]{BCN}). This graph is not a Cayley graph because its halved graphs are the complement of the Hoffman-Singleton graph, which is known not to be a Cayley graph \cite{HoSiNO, Resmini}.

Note finally that if we view the (distance-regular) $\Gamma$ as the incidence graph of a symmetric partial geometric design with parameters $(n,k,\alpha,\beta)$, then it follows that $\alpha=\mu c_3$ and $\beta=(k-1)(\mu-1)$ (we omit a derivation, which is similar as the later derivation in the group case in Proposition \ref{prop:0mupgds}).

\subsection{Partial $\mu$-geometric difference sets}

A partial geometric difference set --- or $1\frac{1}{2}$-difference set --- in a finite group $H$ with parameters $(n,k,\alpha,\beta)$, as introduced by Olmez \cite{O}, is a $k$-subset $D$ of $H$, where $|H|=n$, with the property that every $h \in H$ can be expressed as $d_1d_2^{-1}d_3$, with $d_1,d_2,d_3 \in D$, in either $2k-1+\beta$ or $\alpha$ ways, depending on whether $h \in D$ or not, respectively. Note that the contribution $2k-1$ comes from the ``trivial" ways to express $h \in D$ as a required triple product. As in the case of the usual difference sets, also here the development of a partial geometric difference set $D$ is a symmetric partial geometric design; moreover, $D^{-1}$ is also a partial geometric difference set and its development is the dual design (which has the same parameters).

Given the above characterization of bipartite distance-regular graphs with diameter $4$ among the incidence graphs of symmetric partial geometric designs, it follows easily that the incidence graph of the development of a partial geometric difference set $D$ is distance-regular with $c_2=\mu$ if and only if every $h \in H\setminus\{e\}$ can be written as $d_1d_2^{-1}$, with $d_1,d_2 \in D$ in $\mu$ or $0$ ways and --- dually --- every $h \in H\setminus\{e\}$ can be written as $d_1^{-1}d_2$, with $d_1,d_2 \in D$ in $\mu$ or $0$ ways. We call such a partial geometric difference set a partial $\mu$-geometric difference set.
Note that a difference set is a degenerate case of this; in which case the diameter is $3$ instead of $4$, as we saw before.

We can now obtain a similar characterization as in Proposition \ref{Bipartite distance-regular and difference set}.

\begin{prop}\label{prop:0mupgds}
Let $G$ be a group of order $2n$ and $S$ a subset of $G$ of size $k$. Then the following statements are equivalent:
\begin{enumerate}
\item $S \subseteq G \setminus \{e\}$, $S = S^{-1}$ and $\Cay(G,S)$ is a bipartite distance-regular graph with diameter $4$ and intersection array $\{k,k-1,k-\mu,k-c_3;1,\mu,c_3,k\}$;
\item there is a subgroup $H$ of index $2$ in $G$ such that for every $a \in G \setminus H$, the set $D=Sa^{-1}$ is a  partial $\mu$-geometric difference set with parameters $(n,k,\mu c_3,(k-1)(\mu-1))$ satisfying $D^{-1}=aDa$;
\item there is a subgroup $H$ of index $2$ in $G$ and an element $a \in G \setminus H$ such that the set $D=Sa^{-1}$ is a partial $\mu$-geometric difference set with parameters $(n,k,\mu c_3,(k-1)(\mu-1))$ satisfying $D^{-1}=aDa$. \end{enumerate} \end{prop}

\begin{proof} $(1)\Rightarrow (2)$: Assume (1) and let $H$ be the part of the bipartite graph containing $e$. Let $a\in G\setminus H$ and $D=Sa^{-1}$. Then $D^{-1}=aDa$ and (in group ring notation) we have that $D^{-1}D= aS^{-1}Sa^{-1}=aS^{2}a^{-1}=a(k \{e\}+\mu S_2)a^{-1}=k \{e\}+\mu aS_2a^{-1}$, so every non-identity element of $H$ can be written as $d_1d_2^{-1}$ in $\mu$ or $0$ ways. Dually the same holds because $DD^{-1}= S^2=k \{e\}+\mu S_2$. Finally,
$DD^{-1}D=kD+\mu S_2D=kD+\mu S_2Sa^{-1}=kD+\mu ((k-1)S+c_3S_3)a^{-1}=kD+\mu (k-1)D+ \mu c_3(H \setminus D)= (2k-1)D+(\mu-1)(k-1)D+ \mu c_3(H \setminus D),$ hence $D$ is a partial $\mu$-geometric difference set with parameters $(n,k,\mu c_3,(k-1)(\mu-1))$.

$(2)\Rightarrow (3)$: Trivial.

$(3)\Rightarrow (1)$: Assume (3), then $S$ is inverse-closed because $D^{-1}=aDa$. From above considerations, it follows that the incidence graph of the development of $D$ is a bipartite distance-regular graph $\Gamma$ with diameter $4$ and intersection array $\{k,k-1,k-\mu,k-c_3;1,\mu,c_3,k\}$. This incidence graph is isomorphic to $\Cay(G,S)$. To see this, consider the map $\varphi:\Cay(G,S) \rightarrow \Gamma$ with $\varphi(h)=h$ and $\varphi(a^{-1}h)=Dh$, for $h \in H$. It is clear that $\varphi$ is a bijection. Moreover, if $h_{1} \sim a^{-1}h_{2}$, then $h_1(a^{-1}h_{2})^{-1} \in S=Da$, hence $h_{1}h_{2}^{-1} \in D$, $h_{1}\in Dh_{2}$, and finally $\varphi(h_{1}) \sim \varphi (a^{-1}h_{2})$, and the other way around.
\end{proof}

For example, the Van Lint-Schrijver partial geometry $pg(5,5,2)$ \cite{vLS} can be constructed as the development of a partial $1$-geometric difference set on the additive group of $GF(81)$. Indeed, if $\gamma$ is a primitive fifth root of unity, then $D=\{0,1,\gamma,\gamma^2,\gamma^3,\gamma^4\}$ is such a partial $1$-geometric difference set. Another nice description in the group $\mathbb{Z}_3^4$ is also available \cite[Construction 2]{vLS}.
Hence, by Proposition \ref{lemma:devCayley}, its incidence graph --- the distance-regular Van Lint-Schrijver graph --- is a Cayley graph. The recently constructed other partial geometry $pg(5,5,2)$ \cite{CST,K} does not have a group that acts transitively on the points, so its incidence graph is not a Cayley graph.

Also a few other examples are well known to be Cayley graphs, such as the folded $8$-cube and folded $9$-cube.

Another example that is not a Cayley graph (next to the earlier mentioned graph related to the Hoffman-Singleton graph) is the Leonard graph. This follows because the corresponding design is not self-dual, and hence the Leonard graph is not vertex-transitive \cite[Thm.~11.4.4]{BCN}.

We finally note that in our previous paper \cite[Prop.~3.3]{VJ}, we obtained the condition that $s$ must be $0$ or $4 \mod 6$ for the incidence graph of a generalized quadrangle of order $s$ to be a Cayley graph. Also for $s=4$, it is not a Cayley graph \cite[Prop.~3.7]{VJ}.

\subsection{Symmetric relative difference sets and antipodal graphs}

A special case of partial $\mu$-geometric designs is the family of symmetric transversal designs $STD_{\mu}[r\mu;r]$, or symmetric $(r,\mu)$-nets. Their incidence graphs are precisely the antipodal bipartite distance-regular graphs with 
intersection array $\{r\mu,r\mu-1,(r-1)\mu,1;1,\mu,r\mu-1,r\mu\}$. The corresponding concept of partial $\mu$-geometric difference sets is that of symmetric $(r\mu,r,r\mu,\mu)$-relative difference sets.

Let $H$ be a finite group and $N$ a proper subgroup of $H$ such that $|N|=r$ and $[H:N]=m$. Then a $k$-subset $D$ of $H$ is an $(m,r,k,\mu)$-relative difference set relative to $N$ (the ``forbidden'' subgroup) whenever every $h \in H\setminus\{e\}$ can be written as $d_1d_2^{-1}$, with $d_1,d_2 \in D$ in $0$ or $\mu$ ways, depending on whether $h \in N$ or not, respectively. Moreover, we say that $D$ is symmetric whenever $D^{-1}$ is also a relative difference set (possibly with a different forbidden subgroup). Jungnickel \cite{Ju} showed that if $N$ is normal, then $D$ is symmetric.

Similar as before, the development of a symmetric $(r\mu,r,r\mu,\mu)$-relative difference set is a symmetric transversal design $STD_{\mu}[r\mu;r]$, and hence its incidence graph is an antipodal bipartite distance-regular graph with diameter $4$ (in fact, it is an $r$-cover of a complete multipartite graph).

The following is in some sense a special case of Proposition \ref{prop:0mupgds}.

\begin{prop}\label{propRelDifSet}
Let $G$ be a group of order $2r^2\mu$ and $S$ a subset of $G$. Then the following statements are equivalent:
\begin{enumerate}
\item $S \subseteq G \setminus \{e\}$, $S = S^{-1}$ and $\Cay(G,S)$ is an antipodal bipartite distance-regular graph with diameter $4$ and intersection array $\{r\mu,r\mu-1,(r-1)\mu,1;1,\mu,r\mu-1,r\mu\}$;
\item there is a subgroup $H$ of index $2$ in $G$ and a subgroup $N$ of $H$ of order $r$ such that for every $a \in G \setminus H$, the set $D=Sa^{-1}$ is a symmetric $(r\mu,r,r\mu,\mu)$-relative difference set relative to $N$ in $H$ satisfying $D^{-1}=aDa$;
\item there is a subgroup $H$ of index $2$ in $G$, a subgroup $N$ of $H$ of order $r$, and an element $a \in G \setminus H$ such that the set $D=Sa^{-1}$ is a symmetric $(r\mu,r,r\mu,\mu)$-relative difference set relative to $N$ in $H$ satisfying $D^{-1}=aDa$. \end{enumerate} \end{prop}

\begin{proof} Most of the proof is omitted because it is similar as the proof of Proposition \ref{prop:0mupgds}. Additional details are as follows.

$(1)\Rightarrow (2)$: Let $N=S_4 \cup \{e\}$. Because the Cayley graph is antipodal, it follows that $N$ is a subgroup of $H$. It follows that $D=Sa^{-1}$ is an $(r\mu,r,r\mu,\mu)$-relative difference set relative to $N$ in $H$. Similarly, $D^{-1}$ is an $(r\mu,r,r\mu,\mu)$-relative difference set relative to the subgroup $a^{-1}Na$ in $H$.
\end{proof}

The antipodal bipartite distance-regular graphs with diameter $4$ and $\mu=1$ are incidence graphs of affine planes minus a parallel class of lines.
As we already observed before (but provided with few arguments) in \cite{VJ},
relative difference sets in the abelian group $GF(q)^2$ are known for the Desarguesian affine planes $AG(2,q)$ minus a parallel class of lines, for prime powers $q$. Again, by Proposition \ref{lemma:devCayley}, their incidence graphs, which are antipodal bipartite distance-regular graphs with diameter $4$, are therefore Cayley graphs.

Many more constructions of relative difference sets are known that give rise to distance-regular Cayley graphs. For such constructions, we refer to the survey paper on relative difference sets by Pott \cite{Pott}, for example.

A distinguished case is the one with $r=2$: symmetric transversal designs $STD_{\mu}[2\mu;2]$ give rise to the Hadamard graphs (on $8\mu$ vertices), which in turn are equivalent to (distance-regular) Hadamard matrices of size $2\mu$; see \cite[\S 1.8]{BCN}. The smallest Hadamard graphs are $C_8$ and the $4$-cube, which are clearly Cayley graphs. For $\mu=4$ and $\mu=6$, the (unique) Hadamard graphs are also Cayley graphs, as they can be obtained from appropriate relative difference sets (as in \cite{SchTan}).

We mention once more the distance-regular $4$-cube (recall also the earlier remarks in Section \ref{sec:bipartiteCayley}), for it has a cospectral graph: the Hoffman graph. This graph is neither a Cayley graph nor is it distance-regular. Still, it can be constructed from a square transversal design; however its dual is not a transversal design. In this case this implies that the Hoffman graph is only ``half distance-regular". Likewise, Hiramine \cite{Hi} constructed non-symmetric relative difference sets, among others one with parameters $(12,3,12,4)$. It gives rise to a half distance-regular graph that is cospectral to the Suetake graph (the unique distance-regular graph with intersection array $\{12,11,8,1;1,4,11,12\}$). The latter comes from the unique symmetric transversal design $STD_4[12;3]$, see \cite{Sue}.
We found a related relative difference set in the abelian group $H=\mathbb{Z}_{2} \times \mathbb{Z}_{3} \times \mathbb{Z}_{6}$. Indeed, let
$$D=\{000, 002, 004, 005, 011, 023, 100, 101, 114, 122, 123, 125\}.$$

\begin{prop}\label{prop:sue} The set $D$ is a relative difference set in the abelian group $H=\mathbb{Z}_{2} \times \mathbb{Z}_{3} \times \mathbb{Z}_{6}$, relative to the normal subgroup $N=\mathbb{Z}_3$. The incidence graph of its development is the Suetake graph, which is therefore a Cayley graph.
\end{prop}

Another very interesting antipodal bipartite distance-regular graph with diameter $4$ is the hexacode graph on $36$ vertices \cite[Thm.~13.2.2]{BCN}, which comes from the unique $STD_2[6;3]$. We checked that its automorphism group does not have a regular subgroup, so it is not a Cayley graph. Besides the above two examples with $r=3$, we also mention the Pappus graph (the incidence graph of the affine plane of order $3$ minus a parallel class of lines; see above), and the four graphs coming from an $STD_3[9;3]$. The latter were classified by Mavron and Tonchev \cite{MaTo}. We checked that only the one with the smallest automorphism group is not a Cayley graph (in particular, it is not even vertex-transitive). The distance-transitive one can be constructed from a relative difference set in $\mathbb{Z}_{3}^3$, and the other two in $\mathbb{Z}_{9} \times \mathbb{Z}_{3}$ and $(\mathbb{Z}_{3} \times \mathbb{Z}_{3})\rtimes \mathbb{Z}_{3}$, respectively.

We finish this section by mentioning that we checked also that the distance-transitive graph coming from an $STD_2[8;4]$ is a Cayley graph (in particular, that it can be obtained from a relative difference set in $\mathbb{Z}_{4} \times \mathbb{Z}_{4}\times \mathbb{Z}_{2}$).

\section{Bipartite distance-regular Cayley graphs with diameter $3$}\label{sec:d3}

We now return to bipartite distance-regular graphs with diameter three to go deeper into the more general results in Section \ref{sec:bipartiteCayley}. The following result follows immediately from there.

\begin{prop}\label{prop:bip3semidirect}
Let $\Gamma=\Cay(G,S)$ be a distance-regular Cayley graph with intersection array $\{k,k-1,k-\mu ; 1,\mu,k\}$, and $H$ be the part of this bipartite graph which contains the identity element. Then $\Gamma$ can be constructed on $H \rtimes \mathbb{Z}_{2}$, except possibly when $|H|=0 \mod 4$, $k$ is even, and $H$ is non-abelian.
\end{prop}
Note that we are not claiming that $G$ is isomorphic to $H \rtimes \mathbb{Z}_{2}$.

In the remainder of this section, we study the (possibly) exceptional cases for the above theorem for certain well-known groups. Note that if $n=|H|$, then $k(k-1)=(n-1)\mu$, so if $n$ is even, then $\mu$ must be even as well. So we assume that $|H|=0 \mod 4$, and $k$ and $\mu$ are even.

First, observe that the symmetric group has a unique normal subgroup of index $2$, the alternating group, so also any bipartite Cayley graph on the symmetric group, the part containing the identity element must be the alternating group, and $G$ is isomorphic to $H \rtimes \mathbb{Z}_{2}$.

\subsection{The dihedral group}

Similarly as for the symmetric group, for $n$ odd, the dihedral group $D_{2n}$ has a unique normal subgroup of index $2$, the cyclic group (which is trivially in accordance with the above proposition).

If $n$ is even, then $D_{2n}$ has three normal subgroups $H$ of index $2$ and in each case, the group is the semidirect product of $H$ with $\mathbb{Z}_{2}$  \cite[Thm.~3]{N}. For the case where $H$ is not the cyclic group, we obtain further information below. Note that this may be relevant for the classification of distance-regular Cayley graphs on dihedral groups; see \cite{MP2}.

\begin{prop} \label{dihedral}
Let $\Gamma=\Cay(G,S)$ be a distance-regular Cayley graph with intersection array $\{k,k-1,k-\mu ; 1,\mu,k\}$, with $\mu<k-1$, such that $G$ is a dihedral group of order $2n$, with $n=2m$ even and with cyclic subgroup $C$ of order $n$. Let $H$ be the part of this bipartite graph which contains the identity element. If $H$ is not the cyclic group, then the Cayley graph $\Gamma_C=\Cay(C,S \cap C)$ is the incidence graph of a partial geometric design with parameters $(m,k_1,\alpha,\beta)$ and distinct eigenvalues $\{\pm k_1,\pm \sqrt{k-\mu},0\}$, where $k_1=|S \cap C|$, $k_2=k-k_1$, $\alpha=\frac{\mu(2k_1-k_2)}{2}$, and $2k_1-1+\beta-\alpha=(k_1-k_2)^{2}=k-\mu$.
\end{prop}

\begin{proof}
The normal subgroup $C$ gives rise to an equitable partition with two parts for this graph. Let $k_1=|S \cap C|$, $k_2=k-k_1$, then this partition has quotient matrix
$$\begin{bmatrix}
    k_1 & k_2 \\
    k_2 & k_1
  \end{bmatrix}.$$
Note that $k_1 \neq 0$ because $C \neq H$ and $k_2 \neq 0$ since $\Gamma$ is connected. Furthermore, $(k_1-k_2)^{2}=k-\mu$ since the eigenvalues $k=k_1+k_2$ and $k_1-k_2$ of the quotient matrix are also eigenvalues of the Cayley graph $\Cay(G,S)$  \cite[Lemma~2.3.1]{BH}. It then also follows that $4k_1k_2=k(k-1)+\mu=n\mu$.
Next, we will use that the eigenvalues of Cayley graphs can be expressed as character sums of the underlying group by a theorem of Babai \cite[Thm.~3.1]{Bab} (see also the survey paper \cite{LZ}), and use the relation between character sums of $G$ and those of $C$.

The eigenvalues of the circulant Cayley graph $\Gamma_C$ are $$\lambda_{j}=\sum_{c \in S \cap C }\omega^{jc},$$ where $j=0,1,\ldots,n-1$ and $\omega=e^{\frac{2 \pi i}{n}}$ (for convenience, we consider $S \cap C$ as a subset of $\mathbb{Z}_n$).
Note that $\lambda_0=k_1$, and in addition, it is easy to see that $\lambda_{m+j}=-\lambda_{j}$ for every $j=0,1,\ldots,m-1$.
On the other hand, the irreducible characters of the dihedral group $G$ have degree $1$ or $2$. The eigenvalues $\pm k$ (both have multiplicity $1$) of $\Gamma$ correspond to irreducible characters of degree $1$. Thus, all other irreducible characters must correspond to eigenvalues $\pm \sqrt{k-\mu}$ (and indeed, it is easy to see that also $k_1-k_2$ corresponds to an irreducible character of degree $1$).

Each of the $m-1$ irreducible characters $\psi_j$ of degree $2$ corresponds to two eigenvalues $\lambda_{j1}$ and $\lambda_{j2}$ of $\Gamma$, and it follows from Babai's theorem, the character values, and the fact that $S$ is inverse-closed that
$$\lambda_{j1}+\lambda_{j2}=\sum_{s \in S}\psi_{j}(s)=\sum_{c \in S \cap C} (\omega^{jc}+\omega^{-jc})=2\sum_{c \in S \cap C} \omega^{jc}=2 \lambda_j,$$
for $j=1,\dots,m-1$.

Next, we use that $\lambda_{j1}$ and $\lambda_{j2}$ can only take values $\pm \sqrt{k-\mu}$, so $\lambda_j$ can only take values $\pm \sqrt{k-\mu}$ and $0$.
This implies that the bipartite Cayley graph $\Gamma_C$ has (distinct) eigenvalues $\{\pm k_1,0\}$, $\{\pm k_1,\pm \sqrt{k-\mu}\}$, or $\{\pm k_1,\pm \sqrt{k-\mu},0\}$.

If it has eigenvalues $\{\pm k_1,0\}$, then it is a complete bipartite graph. In this case $4mk_2=4k_1k_2=n\mu=2m\mu$ and therefore $\mu=2k_2$. On the other hand, we have $k(k-1)=(n-1)\mu$. Hence $(m+k_2)(m+k_2-1)=(2m-1)2k_2$ and therefore $k_2=m-1$ or $k_2=m$, which implies that $\Gamma$ is a complete bipartite graph minus a perfect matching or a complete bipartite graph, which is a contradiction.

If $\Gamma_C$ has eigenvalues $\{\pm k_1,\pm \sqrt{k-\mu}\}$, then it must be distance-regular \cite[Prop.~15.1.3]{BH}. On the other hand, distance-regular circulant graphs have been classified \cite{MP1} and therefore $\Gamma_C$ must be a complete bipartite graph minus a perfect matching. This implies (from the eigenvalues) that $\mu=k-1$, and so $\Gamma$ is also a complete bipartite graph minus a perfect matching, which again is a contradiction.

Thus, $\Gamma_C$ has eigenvalues $\{\pm k_1,\pm \sqrt{k-\mu},0\}$ and so it is the incidence graph of a partial geometric design with parameters $(m,k_1,\alpha,\beta)$, where $\alpha$ and $\beta$ are as stated
(see \cite[\S3.1]{VS} and Section \ref{PGDS}), which completes the proof.
\end{proof}

As a side remark we note that that circulant distance-regular graphs have been classified \cite{MP1} (and the only non-trivial examples are Paley graphs), and hence it follows that the induced graphs on $C$ cannot be distance-regular.

\subsection{The dicyclic group}

Let $G$ be the dicyclic group $Q_{4m}=\seq{a,b \mid a^{2m}=e,a^{m}=b^{2},b^{-1}ab=a^{-1}}$, and let $n=2m$. If $m$ is odd, then the group $G$ has a unique subgroup of index $2$, the cyclic group $H=\seq{a}$ \cite[Thm.~3]{N}. In this case, we claim that there is no non-trivial bipartite distance-regular Cayley graph with diameter $3$. Indeed, suppose that $ba^{i} \in S$ for some $i$. Then also $b^{-1}a^{i} =(ba^{i})^{-1}\in S$. But this implies that the two vertices $e$ and $b^2$ have the same neighbours, which is a contradiction.
Therefore we can conclude the following.
\begin{prop} Let $m$ be odd.
Then there is no non-trivial bipartite distance-regular Cayley graph with diameter $3$ on the dicyclic groups $Q_{4m}$.
\end{prop}
For $m$ even, the group $Q_{4m}$ has three subgroups of index $2$, and these are cyclic or dicyclic \cite[Thm.~3]{N}. If $H$ is cyclic, then the same argument as above implies that there is no non-trivial bipartite distance-regular Cayley graph with diameter $3$ on the group $G$. If $H$ is dicyclic, then a similar argument as in the proof of Proposition \ref{dihedral} applies. Therefore we can conclude the following.
\begin{prop} \label{dicyclic}
Let $\Gamma=\Cay(G,S)$ be a distance-regular Cayley graph with intersection array $\{k,k-1,k-\mu ; 1,\mu,k\}$, with $\mu<k-1$,
 such that $G$ is the dicyclic group $Q_{4m}$ with cyclic subgroup $C$ of order $n=2m$. Let $H$ be the part of this bipartite graph which contains the identity element. If $H$ is not the cyclic group, then the Cayley graph $\Gamma_C=\Cay(C,S \cap C)$ is the incidence graph of a partial geometric design with parameters $(m,k_1,\alpha,\beta)$ and distinct eigenvalues $\{\pm k_1,\pm \sqrt{k-\mu},0\}$, where $k_1=|S \cap C|$, $k_2=k-k_1$, $\alpha=\frac{\mu(2k_1-k_2)}{2}$, and $2k_1-1+\beta-\alpha=(k_1-k_2)^{2}=k-\mu$.
\end{prop}

We also note that if $m$ is even and $H$ is a dicyclic group, then the only involution of the group $G=Q_{4m}$ is in $H$. So $S$ contains no involutions, which implies that the Cayley graph $\Cay(G,S)$ is {\em not} on the semidirect product of a group --- the part of this bipartite Cayley graph which contains the identity element --- and $\mathbb{Z}_{2}$.

\subsection{The semidihedral group}

Let $m=2^{\ell-1}$, with $\ell > 1$. Let $G$ be the semidihedral group $SD_{4m}=\seq{a,b \mid a^{2m}=b^{2}=e,bab=a^{m-1}}$. Then the group $G$ has three subgroups of index $2$, the cyclic group, the dihedral group $D_{2m}$, and the dicyclic group $Q_{2m}$ \cite[Thm.~3]{N}. If $H$ is cyclic or dicyclic, then there exist involutions outside the normal subgroup $H$ and therefore the group $G$ is the semidirect product of $H$ with $\mathbb{Z}_{2}$.

The case that $H$ is a dihedral group is related to an old problem about the (non)-existence of non-trivial difference sets in dihedral groups. It is proved in \cite{L} that if an $(n,k,\mu)$-difference set in a dihedral group of order $n$ exists, then $k-\mu$ must be odd. Thus, Lemma \ref{oddparameters} can be applied, and we obtain the following.

\begin{prop}
Every bipartite distance-regular Cayley graph with diameter $3$ on a semidihedral group is on the semidirect product of a group --- the part of this bipartite Cayley graph which contains the identity element --- and $\mathbb{Z}_{2}$.
\end{prop}
We note that the character table of the semidihedral group is more complicated than the character tables of the dihedral group and the dicyclic group, so to extract a similar result as Propositions \ref{dihedral} and \ref{dicyclic} seems quite ambitious.

\subsection{Small cases}

We finish this section with some interesting small cases. The smallest $2$-$(n,k,\mu)$ designs with $n$ and $k$ even and $k<n-1$ are the $2$-$(16,6,2)$ designs, and their complementary $2$-$(16,10,6)$ designs. The next smallest ones are the $2$-$(64,28,12)$ designs.

There are three non-isomorphic $2$-$(16,6,2)$ designs and therefore three non-isomorphic distance-regular graphs with intersection array $\{6,5,4;1,2,6\}$ \cite[p.~222]{BCN}. It turns out that all three are Cayley graphs. One of them is the folded $6$-cube which is a Cayley graph on the elementary abelian $2$-group of order $32$ \cite[\S3.3]{VJ}. By using GAP \cite{G} and the difference sets in \cite[Table~18.77]{CD}, we can construct all three as follows.

\begin{itemize}
\item $G_{1}$ is the elementary abelian group of order $32$ with generators $\seq{a,b,c,d,f}$ of involutions, $H=\seq{a,b,c,d}$, and $S=\{af,bf,cf,df,f,abcdf\}$;
\item $G_{2}=\seq{a,b,c \mid a^{8}=b^{2}=c^{2}=e,ab=ba,cac=a^{-1},cbc=b^{-1}}$ is isomorphic to $\Dih(H)$, $H=\seq{a,b}$ is isomorphic to $\mathbb{Z}_{8} \times \mathbb{Z}_{2}$, and $S=\{c,ca,ca^{2},ca^{4},cab,ca^{6}b\}$;
\item $G_{3}=\seq{a,b,c,d \mid a^{4}=c^{2}=d^{2}=(da)^{2}=(da^{2})^{2}=e,a^{2}=b^{2},ac=ca,bc=cb,dabc=(db)^{-1}}$ is isomorphic to $SD_{16} \times \mathbb{Z}_{2}$, $H=\seq{a,b,c}$ is isomorphic to $Q_{8} \times \mathbb{Z}_{2}$ and
    $S=\{d,da^{2},(db)^{4}da, \break(db)^{6}da,db,dabc\}.$
\end{itemize}
This implies that every bipartite distance-regular graph with diameter $3$  on $32$ vertices is a Cayley graph on the semidirect product of a group --- the part of this bipartite Cayley graph which contains the identity element --- and $\mathbb{Z}_{2}$. It turns out that all three distance-regular graphs have the $4$-cube as an induced subgraph. This antipodal bipartite distance-regular graph with diameter $4$ can (once more) be constructed as a Cayley graph $\Cay(K,S \cap K)$ on a subgroup $K$ of index $2$ --- cf.~Propositions \ref{dihedral} and \ref{dicyclic} --- as described below (see also the structures of connected $4$-regular bipartite integral Cayley graphs in \cite[Table~3]{MW}). Recall also Proposition \ref{propRelDifSet}.
\begin{itemize}
\item $K=\seq{af,bf,cf,df}$ is an elementary abelian subgroup in $G_{1}$, and $S \cap K=\{af,bf,cf,df\}$;
\item $K=\seq{a^{2},b,c} \cong D_{8} \times \mathbb{Z}_{2}$ is a subgroup in $G_{2}$, and $S \cap K=\{c,ca^{2},ca^{4},ca^{6}b\}$;
\item $K=\seq{db,da} \cong SD_{16}$ is a subgroup in $G_{3}$, and $S \cap K=\{(db)^{4}da,(db)^{6}da,db,dabc\}$.
\end{itemize}

\section{Bipartite distance-regular Cayley graphs with larger diameter}\label{sec:largerdiameter}

We conclude this paper with some (mostly known) results on bipartite distance-regular Cayley graphs with larger diameter. We note that the list of known bipartite distance-regular graphs with diameter at least $5$ is quite limited, see \cite[\S 6.11]{BCN}.

The families of bipartite distance-regular graphs with unbounded diameter $d$ include the $2d$-cycle, the $d$-cube, and the folded $2d$-cube, which are clearly Cayley graphs. In an earlier paper \cite[Prop.~3.1]{VJ}, we showed that the Doubled Odd graphs are not Cayley graphs.

It was also shown that the Foster graph (with diameter $8$) is not a Cayley graph \cite[Prop.~4.2]{VJ}, nor are the incidence graphs of the known generalized hexagons (of order any prime power) \cite[Prop.~3.6]{VJ}. The latter are bipartite distance-regular graphs with diameter $6$.

Finally, we consider the (known) antipodal bipartite distance-regular graphs with diameter $5$. These are all bipartite doubles of triangle-free strongly regular graphs $\Gamma$. Because the halved graphs of such a bipartite double is isomorphic to the complement of $\Gamma$, it follows that the bipartite double of $\Gamma$ is a Cayley graph if and only if $\Gamma$ is a Cayley graph (cf.~\cite[\S3.4]{VJ}). The known examples of triangle-free strongly regular graphs that are well known not to be Cayley graphs are the Petersen graph and the Hoffman-Singleton graph \cite{Resmini}. We also checked that the Gewirtz graph and the $M_{22}$-graph on $77$ vertices are not Cayley graphs, because their automorphism groups do not have subgroups of order $n$ (the number of vertices). Besides the folded $5$-cube (whose bipartite double is the $5$-cube), the only (known) example in this class that is a Cayley graph is the Higman-Sims graph \cite{Ha, JK}.

\section*{Acknowledgements}

\noindent The research of Mojtaba Jazaeri was in part supported by a grant from School of Mathematics, Institute for Research in Fundamental Sciences (IPM) (No. 96050014). Mojtaba Jazaeri is grateful to the Research Council of Shahid Chamran University of Ahvaz for financial support (SCU.MM1400.29248). The authors thank Robert Bailey for \cite{drgorg} and a referee for some useful comments.

\end{document}